\documentclass[]{amsart}

\usepackage[margin=1in]{geometry}
\usepackage[utf8]{inputenc}
\usepackage{amsmath,amssymb,amsthm,mathtools,mathrsfs}
\usepackage{enumitem}
\usepackage{hyperref}
\usepackage{orcidlink}
\newtheorem{thm}{Theorem}[section]

\newtheorem*{fixed point criterion}{\fixed point criterion}
\newtheorem{cor}[thm]{Corollary}
\newtheorem{lem}[thm]{Lemma}
\newtheorem{prop}[thm]{Proposition}

\theoremstyle{definition}
\newtheorem{defn}[thm]{Definition}

\newtheorem{ques}[thm]{Question}

\theoremstyle{remark}

\numberwithin{equation}{section}

\newcommand{\F}{\mathbb F}

\newcommand{\Z}{\mathbb Z}
\newcommand{\GL}{\operatorname{GL}}
\newcommand{\UT}{\operatorname{UT}}
\newcommand{\Aut}{\mathrm{Aut}}
\newcommand{\Out}{\mathrm{Out}}

\newcommand{\dgen}{\mathrm{rk}}

\title{Finitely generated Howson groups which are not strongly Howson}

\author[Ke Wang]{Ke Wang\orcidlink{0009-0005-7108-3725}}
	\address{School of Mathematics and Statistics, Xi'an Jiaotong University, Xi'an 710049, P. R. China}
	\email{keqiyehuopo@stu.xjtu.edu.cn}
	
	\author[Qiang Zhang]{Qiang Zhang\orcidlink{0000-0001-6332-5476}}
	\address{School of Mathematics and Statistics, Xi'an Jiaotong University, Xi'an 710049, P. R. China}
	\email{zhangq.math@mail.xjtu.edu.cn}
    
\date{\today}
\thanks{The authors are partially supported by NSFC (No. 12471066).}

\begin{document}

\begin{abstract}
    A group $G$ is called a Howson group if the intersection of any two finitely generated subgroups of $G$ is again finitely generated. It is called strongly Howson if, in addition, the rank of such an intersection is bounded only in terms of the ranks of the two subgroups. Strongly Howson groups are indeed Howson. Recently, Zhang and Zhao showed that the converse fails, by constructing the first examples of infinitely generated Howson groups which are not strongly Howson. In this note, we provide finitely generated examples with and without torsion elements, thereby answering a question posed by Zhang and Zhao.
\end{abstract}

\maketitle

\section{Introduction}

For a group $G$, let $\dgen(G)$ denote the \emph{rank}, i.e., the minimum cardinality of a generating set of $G$. For a group $G$ and positive integers $m,n$, let
$$\xi_G(m,n)= \sup\{~\dgen(H\cap K)\mid H,K\leq G,~ \dgen(H)\leq m,~ \dgen(K)\leq n~\},$$
where the supremum is allowed to be infinite. Since all subgroups of a cyclic group are cyclic, it follows that $\xi_G(1,n)=\xi_G(m,1)=1$.  Furthermore, we have:

\begin{defn}
A group $G$ is called a \emph{Howson group} if $H\cap K$ is finitely generated whenever $H,K\leq G$ are finitely generated. The group $G$ is called \emph{strongly Howson} if $\xi_G(m,n)<\infty$ for all positive integers $m,n$.
\end{defn}

The Howson property originates in Howson's theorem that the intersection of two finitely generated subgroups of a free group is again finitely generated \cite{Ho54}. Hanna Neumann subsequently gave uniform rank estimates for these intersections \cite{Ne56}. The modern strengthened form of the Hanna Neumann conjecture was proved independently by Friedman and Mineyev \cite{Fr14,Mi12}. Similar intersection questions have also been studied for free products, amalgamated products, HNN extensions and hyperbolic groups; see, for instance, \cite{Ba66,BB80,Ka97,Ar98,AMS14}.

The notion of a strongly Howson group was isolated explicitly by Ara\'ujo, Silva and Sykiotis \cite{ASS15} in their study of finite extensions. Clearly every strongly Howson group is Howson. Recently, Zhang and Zhao \cite{ZZ26} proved that the converse is false by constructing Howson groups which are not strongly Howson. As their examples fail to be finitely generated, they raised the following question.

\begin{ques}[Zhang-Zhao, 2024]\label{Question f.g}
Is there a finitely generated or finitely presented Howson group which is not strongly Howson?
\end{ques}

In this note, we answer this question affirmatively via the theorem below. Consequently, finite generation does not ensure that the Howson property implies the strongly Howson property.

\begin{thm}\label{thm:main}
There exists an infinite group $\mathcal{G}$ satisfying all the following properties: 
\begin{enumerate}
    \item $\mathcal{G}$ is a Howson but not strongly Howson group;
    \item $\mathcal{G}$ is a simple and $2$-generated group;
    \item $\mathcal{G}$ is a group with nontrivial torsion elements;
\end{enumerate}
\end{thm}
In fact, the group $\mathcal G$ constructed in Theorem \ref{thm:main} fails to be virtually torsion-free or residually finite, and hence fails to be linear over any field (see Corollary \ref{cor:non-residual finiteness}).
Furthermore, note that the group $\mathcal{G}$ contains abundant torsion elements, as do the groups from \cite{ZZ26}. This naturally leads to the following question: Does there exist a torsion-free Howson group which is not strongly Howson? Our next theorem provides an affirmative answer.

\begin{thm}\label{thm:torsionfree}
There exists a $2$-generated, torsion-free, simple group $\mathcal G'$ which is Howson but not strongly Howson.  More precisely, $\xi_{\mathcal G'}(m,k)=\infty$ for any integers $m, k\geq 2$.
\end{thm}

To prove Theorem~\ref{thm:main}, we first construct an infinite family of finite groups $D_n$ of odd order. Each group contains two $2$-generated subgroups whose intersection has rank $n$. We then embed this family into a $2$-generated simple group whose proper subgroups are tightly controlled, using Obraztsov's embedding construction \cite{Ob90}.  The proof of Theorem~\ref{thm:torsionfree} follows the same strategy, but replaces the finite elementary abelian intersections by free abelian groups and apply another embedding theorem of Obraztsov for groups without involutions \cite{Ob96}.


\section{The torsion construction\label{sect 2}}

\subsection{Finite groups with large intersections}

For every integer $n\geq 3$, choose an odd prime $p_n>n$ such that $p_m\neq p_n$ for any $m\neq n$. Let
\begin{equation}\label{eq. Vn}
  V_n=\F_{p_n}^n.  
\end{equation}
We regard $V_n$ as an elementary abelian $p_n$-group of rank $n$. Let $e_1,\ldots,e_n$ be the standard basis.  Define a nilpotent linear transformation $J_n$ by
$$J_ne_i=e_{i+1}\quad (1\leq i<n), \qquad J_ne_n=0.$$
Thus $J_n^n=0$ and $J_n^{n-1}\neq 0$. Denote
$$P_n=I+J_n,\qquad Q_n=I+J_n+J_n^2$$
in $\GL(V_n)$.

\begin{lem}\label{lem:linear}
The elements $P_n$ and $Q_n$ both have order $p_n$, and
$\langle P_n\rangle\cap \langle Q_n\rangle=1$
inside $\GL(V_n)$.
\end{lem}

\begin{proof}
Write $p=p_n$ and $J=J_n$. Since $p>n$, a straightforward calculation shows that $J^p=(J+J^2)^p=0$. In characteristic $p$, we have
$$(I+J)^p=I+J^p=I,\qquad (I+J+J^2)^p=I+J^p+J^{2p}=I.$$
Thus the orders of $P_n$ and $Q_n$ divide $p$. If $1\leq a<p$, then the coefficient of $J$ in $(I+J)^a$ is $a\neq 0$ in $\F_p$, so $(I+J)^a\neq I$. The same coefficient argument applied to $(I+J+J^2)^a$ shows that $Q_n^a\neq I$. Hence both orders are exactly $p$.

It remains to prove that the two cyclic subgroups meet trivially. Suppose
$$ P_n^a=Q_n^b$$
for some $a,b\in\F_p$, where we identify exponents modulo $p$. The elements $I,J,\ldots,J^{n-1}$ are linearly independent in the algebra $\F_p[J]$, so we may compare coefficients of powers of $J$.  Comparing the coefficient of $J$ gives $a=b$. The coefficient of $J^2$ in
$P_n^a=(I+J)^a$ is $\binom{a}{2}, $ 
while the coefficient of $J^2$ in $Q_n^b=(I+J+J^2)^b$ is $b+\binom{b}{2}.$  Since $a=b$, equality of the two coefficients gives $b=0$ in $\F_p$, and hence $a=0$. Therefore $P_n^a=Q_n^b=I$, as required.
\end{proof}

We now denote
$$ L_n=\langle P_n,Q_n\rangle\leq \GL(V_n).$$
Since $P_n$ and $Q_n$ are upper unitriangular matrices, $L_n$ is a subgroup of the unitriangular group $\UT_n(\F_{p_n})$, and hence is a finite $p_n$-group. Furthermore, define the finite semidirect product
\begin{equation}\label{eq. Dn}
  D_n:=V_n\rtimes L_n.  
\end{equation}
Then we have:

\begin{prop}\label{prop:finitepieces}
For every $n\geq 3$, the group $D_n$ in Eq. (\ref{eq. Dn}) is a $p_n$-group of odd order containing subgroups $A_n$ and $B_n$ such that
$$ \dgen(A_n)\leq 2,\qquad \dgen(B_n)\leq 2,\qquad A_n\cap B_n=V_n,\qquad \dgen(A_n\cap B_n)=n.$$
\end{prop}

\begin{proof}
Since both $V_n$ and $L_n$ are finite $p_n$-groups, $D_n=V_n\rtimes L_n$ is also a finite $p_n$-group. Inside $D_n$, define
\begin{equation}\label{eq. An Bn}
   A_n=\langle e_1,P_n\rangle, \qquad B_n=\langle e_1,Q_n\rangle. 
\end{equation} 
We first show that $e_1$ generates the whole normal subgroup $V_n$ under the actions of both $P_n$ and $Q_n$. For $P_n$ this is immediate from $P_n-I=J_n$, because
$$ e_1,\ J_ne_1,\ J_n^2e_1,\ldots,\ J_n^{n-1}e_1$$
is precisely the basis $e_1,\ldots,e_n$.

For $Q_n$, denote $R_n=Q_n-I=J_n+J_n^2=J_n(I+J_n)$. Since $J_n$ commutes with $I+J_n$,
$$ R_n^k e_1=J_n^k(I+J_n)^k e_1 =e_{k+1}+\mathbf{v}_k\cdot (e_{k+2}~,\ldots,~e_n)^{\mathrm T}$$
for $0\leq k\leq n-1$ and $\mathbf{v}_k\in\F_{p_n}^{n-k-1}$. Hence the vectors
$$ e_1,\ R_ne_1,\ R_n^2e_1,\ldots,\ R_n^{n-1}e_1$$
form a basis of $V_n$ by triangularity. It follows that
$$ A_n=V_n\rtimes \langle P_n\rangle, \qquad B_n=V_n\rtimes \langle Q_n\rangle.$$

Let $\pi:D_n=V_n\rtimes L_n\to L_n$ be the natural projection.  If $x\in A_n\cap B_n$, then
$$ \pi(x)\in \langle P_n\rangle\cap \langle Q_n\rangle.$$
By Lemma~\ref{lem:linear}, this intersection is trivial. Therefore $x\in \ker\pi=V_n$. Conversely $V_n$ is contained in both $A_n$ and $B_n$, and so
$$ A_n\cap B_n=V_n.$$
Now combining Eq. (\ref{eq. Vn}) with Eq. (\ref{eq. An Bn}), we have obtained $$\dgen(A_n)\leq 2, \quad\dgen(B_n)\leq 2, \quad \dgen(A_n\cap B_n)=n.$$
This completes the proof.
\end{proof}

\subsection{Obraztsov's embedding theorem I}


The following embedding theorem from \cite[Theorem]{Ob90} plays a key role in the proof of Theorem \ref{thm:main}.

\begin{thm}[Obraztsov, \cite{Ob90}]\label{Ob:embedding}
Let $\mathscr{G}=\{G_i\}_{i\in I}$ be a countable family of nontrivial groups of finite or at most countable cardinality without involution, and let $q$ be an arbitrary sufficiently large fixed odd integer (for example $q>2\cdot 10^{77}$). Then there exists a countable group $G$ such that
\begin{enumerate}
\item every $G_i\in\mathscr{G}$ can be embedded in $G$ such that $G_i\cap G_j=\{1\}$ whenever $i\neq j$;
    \item $G$ is a simple group;
    \item $G$ is $2$-generated by any $d_1$ and $d_2$ with $d_1\in G_i\backslash\{1\}$ for $G_i\in\mathscr{G}$ and $d_2\in G\backslash G_i$;
    \item any proper subgroup of $G$ is either a cyclic group whose order divides $q$ or else is conjugate into some group $G_i\in\mathscr{G}$. 
    \item any two distinct maximal subgroups of $G$ intersect trivially.
\end{enumerate}
\end{thm}
Recall that a group element is called an \emph{involution} if it is of order $2$.


\subsection{Proof of Theorem~\ref{thm:main}}
 By Proposition \ref{prop:finitepieces}, every $D_n$ is a 
$p_n$-group of odd order for any $n\geq 3$, so the family $\mathscr{D}:=\{D_i\}_{i=3}^\infty$ is a countable family of nontrivial finite groups without involutions. Applying Theorem \ref{Ob:embedding} to the family $\mathscr{D}$, we obtain a group, denoted by $\mathcal G$, for Theorem \ref{thm:main}. 

\begin{prop}\label{prop:howson}
The group $\mathcal G$ is Howson.
\end{prop}

\begin{proof}
Let $H,K\leq \mathcal G$ be two finitely generated subgroups. 
If $H=K=\mathcal G$, then $H\cap K=\mathcal G$, which is 2-generated by Theorem \ref{Ob:embedding}(3).  
We now assume that one of $H$ and $K$ is a proper subgroup of $\mathcal G$. By Theorem~\ref{Ob:embedding}(4), each proper subgroup of $\mathcal G$, in particular $H\cap K$, is either cyclic of finite order, or is contained in a conjugate of one of the finite groups $D_n$.  Hence $H\cap K$ is finitely generated, and therefore $\mathcal G$ is Howson.
\end{proof}


\begin{prop}\label{prop:notstrong}
The group $\mathcal G$ is not strongly Howson.  In fact, $\xi_\mathcal G(m,k)=\infty$ for any integers $m, k\geq 2$.
\end{prop}

\begin{proof}
For every $n\geq 3$, the embedded copy of $D_n$ contains the subgroups
$$ A_n=\langle e_1,P_n\rangle, \qquad B_n=\langle e_1,Q_n\rangle$$
from Proposition \ref{prop:finitepieces}.  These subgroups are $2$-generated, and their intersection in $G$ is the same as their intersection in $D_n$, namely
$$ A_n\cap B_n=V_n.$$
Therefore
$$ \dgen(A_n\cap B_n)=\dgen(V_n)=n.$$
Since $n$ is arbitrary, the ranks of intersections of two $2$-generated subgroups of $\mathcal G$ are unbounded.  Thus $\xi_\mathcal G(m,k)\geq \xi_\mathcal G(2,2)=\infty$  for any integers $m, k\geq 2$.
\end{proof}

\begin{proof}[\textbf{Proof of Theorem~\ref{thm:main}}]
Theorem~\ref{Ob:embedding} ensures that $\mathcal G$ is an infinite, $2$-generated simple group with nontrivial torsion elements. Propositions~\ref{prop:howson} and \ref{prop:notstrong} imply that $\mathcal G$ is Howson but not Strongly Howson.
\end{proof}

\subsection{Residual finiteness, torsion, and linearity}

\begin{cor}\label{cor:non-residual finiteness}
The group $\mathcal G$ has no proper subgroup of finite index (and hence has no nontrivial finite quotient). In particular, $\mathcal G$ is not virtually torsion-free, not residually finite,  and hence not linear over any field.
\end{cor}

\begin{proof}
Suppose first that $\mathcal G$ has a proper subgroup $M$ of finite index. Since $\mathcal G$ is $2$-generated and infinite, the normal core
$$N:=\bigcap_{g\in \mathcal G}gMg^{-1}$$
of $M$ is a nontrivial, proper, normal subgroup of finite index in $\mathcal{G}$, contradicting that $\mathcal G$ is simple. Therefore $\mathcal G$ has no proper subgroup of finite index.  Consequently, $\mathcal G$ is neither virtually torsion-free nor residually finite. By Mal’cev’s result, all finitely generated linear groups are residually finite (see \cite{Ma40} or \cite{Ni13}). Therefore, $\mathcal G$ is not linear over any field.
\end{proof}




\section{The torsion-free construction}\label{sec:torsionfree}

We now give a torsion-free variant. The finite vector spaces used above are replaced by free abelian groups, and Obraztsov's embedding theorem (Theorem \ref{Ob:embedding}) is replaced by Obraztsov's relative embedding theorem (Theorem \ref{bb:obraztsovC}).

\subsection{Torsion-free pieces with large intersections}

For every integer $n\geq 3$, let $W_n:=\Z^n$ with standard basis $e_1,\ldots,e_n$. Let $J_n$ be the integral nilpotent Jordan block
$$ J_ne_i=e_{i+1}\quad (1\leq i<n), \qquad J_ne_n=0.$$
Define
$$ \widehat P_n=I+J_n,\qquad \widehat Q_n=I+J_n+J_n^2$$
in $\GL_n(\Z)$.

\begin{lem}\label{lem:tflinear}
The cyclic subgroups $\langle \widehat P_n\rangle$ and $\langle \widehat Q_n\rangle$ are infinite cyclic and intersect trivially. In particular,
$$ \widehat L_n:=\langle \widehat P_n,\widehat Q_n\rangle\cong \Z^2.$$
\end{lem}

\begin{proof}
The matrices $\widehat P_n$ and $\widehat Q_n$ are polynomials in $J_n$, so they commute. For every integer $a$ one has, modulo the ideal generated by $J_n^3$,
$$ (I+J_n)^a=I+aJ_n+\binom a2J_n^2,$$
where the usual polynomial formula for $\binom a2$ is valid for negative integers as well. Similarly, for every integer $b$,
$$ (I+J_n+J_n^2)^b =I+bJ_n+\left(b+\binom b2\right)J_n^2 \pmod{J_n^3}.$$
These formulae show first that $\widehat P_n$ and $\widehat Q_n$ have infinite order: the coefficient of $J_n$ is nonzero unless the exponent is zero.

Suppose now that $\widehat P_n^a=\widehat Q_n^b$ with $a,b\in\Z$.  Comparing the coefficient of $J_n$ gives $a=b$. Comparing the coefficient of $J_n^2$ then gives
$$ \binom a2=a+\binom a2,$$
and hence $a=0$. Here we use that $I,J_n,J_n^2$ are linearly independent in the subring $\Z[J_n]\leq M_n(\Z)$ for $n\geq 3$.  Thus $b=0$ as well. Therefore the two infinite cyclic subgroups intersect trivially. Since the two generators commute, we obtain $\widehat L_n\cong \Z^2$.
\end{proof}

We now define
$$\widehat D_n:=W_n\rtimes \widehat L_n\cong \Z^n\rtimes \Z^2.$$
This is a torsion-free polycyclic group. Indeed, both $W_n$ and $\widehat L_n$ are torsion-free, and any torsion element of $\widehat D_n$ would project to a torsion element of $\widehat L_n$ and then lie in $W_n=\Z^n$. Furthermore, we have:

\begin{prop}\label{prop:tfpieces}
For every $n\geq 3$, the group $\widehat D_n$ is countable, torsion-free and polycyclic, and it contains subgroups $\widehat A_n$ and $\widehat B_n$ such that
$$ \dgen(\widehat A_n)\leq 2,\qquad \dgen(\widehat B_n)\leq 2,\qquad \widehat A_n\cap \widehat B_n=W_n\cong \Z^n.$$
In particular, $\dgen(\widehat A_n\cap \widehat B_n)=n$.
\end{prop}

\begin{proof}
Define
$$ \widehat A_n=\langle e_1,\widehat P_n\rangle, \qquad \widehat B_n=\langle e_1,\widehat Q_n\rangle.$$
As in Proposition~\ref{prop:finitepieces}, the element $e_1$ generates all of $W_n$ under the action of $\widehat P_n$, because $\widehat P_n-I=J_n$. For $\widehat Q_n$, put $\widehat R_n=\widehat Q_n-I=J_n+J_n^2$. The vectors
$$ e_1,\ \widehat R_ne_1,\ \widehat R_n^2e_1,\ldots,\ \widehat R_n^{n-1}e_1$$
form a basis of $W_n$ by triangularity. Therefore
$$ \widehat A_n=W_n\rtimes \langle \widehat P_n\rangle, \qquad \widehat B_n=W_n\rtimes \langle \widehat Q_n\rangle.$$
Projecting $\widehat D_n$ to $\widehat L_n$ and using Lemma~\ref{lem:tflinear}, we obtain
$$ \widehat A_n\cap \widehat B_n=W_n\cong \Z^n.$$
Therefore, $\dgen(\widehat A_n\cap \widehat B_n)=n$.
\end{proof}

\subsection{Obraztsov's embedding theorem II}

We shall use the following form of Obraztsov's later embedding theorem \cite[Theorem C]{Ob96}. 

\begin{thm}[Obraztsov, \cite{Ob96}]\label{bb:obraztsovC}
Let $\{G_i\}_{i\in I}$ be a countable set of nontrivial countable groups containing either three groups or two groups of which one has order at least 3, and let $H$ be an arbitrary countable (for example, trivial) group. Then the free amalgam $\Omega^1$ of the groups $H$ and $G_i, i\in I$, can be embedded in a group $G=\langle\Omega\rangle$, where $\Omega=\Omega^1\backslash\{1\}$, with the following properties:
\begin{enumerate}
    \item the free amalgam of the groups $G_i$ is embedded in a simple normal infinite subgroup $L=\langle\Omega\backslash H\rangle$ of $G$ and $G$ is the semidirect product of $H$ and $L$;
    \item $\Aut(L)\cong G$ (and $\Out(L)\cong H$) and for each $g\in H\backslash\{1\}$, $g$ is a regular automorphism of $L$;
    \item if $x,y\in L$ with $x\in G_i\backslash\{1\}, y\notin G_i$ for some $i\in I$, then either $L$ is generated by the pair $(x,y)$ or $x$ and $y$ are involutions, or $x$ and $xy$ are involutions in $G$;
    \item  every proper subgroup of $L$ is either infinite cyclic or infinite dihedral (if one of the groups $G_i, i \in I$, or $H$ has involutions), or contained in a subgroup conjugate in $G$ to some $G_i, i \in I$.
\end{enumerate}
\end{thm}

If in addition, $H$ is trivial, we can immediately obtain $G=L$ in Theorem \ref{bb:obraztsovC}. Furthermore, if each group $G_i,i\in I$ is torsion-free and contains no involutions, we have the following corollary.

\begin{cor}\label{cor ob embedding thm}
 Let $\{G_i\}_{i\in I}$ with $|I|\geq 3$ be a countable set of nontrivial torsion-free groups without involutions. Then, there exists a group $G$ with the following properties:
 \begin{enumerate}
 \item $G$ is simple, and every $G_i,i\in I$ can be embedded in $G$;
     \item $G$ is torsion-free;
     \item if $x,y\in G$ with $x\in G_i\backslash\{1\}, y\notin G_i$ for some $i\in I$, then $G$ is $2$-generated by the pair $(x,y)$;
     \item every proper subgroup of $G$ is either infinite cyclic or contained in a subgroup conjugate in $G$ to some $G_i, i \in I$.
 \end{enumerate}
\end{cor}

\begin{proof}
Items (1), (3) and (4) follow immediately from the corresponding items of Theorem \ref{bb:obraztsovC}. It remains to prove item (2). Indeed, if $x\in G$ has finite order, then $\langle x\rangle$ is a proper finite cyclic subgroup of $G$. Then by item (4), $\langle x\rangle$ is contained in a conjugate of some $G_i$, it follows that $x=1$ because $G_i$ is torsion-free and can be embedded in $G$. Therefore, the resulting group $G$ is again torsion-free.  
\end{proof}

\subsection{Proof of Theorem \ref{thm:torsionfree}}
Now apply Corollary \ref{cor ob embedding thm} to the family
$$ \{\widehat D_3,\widehat D_4,\widehat D_5,\ldots\}.$$
We identify each $\widehat D_n$ with its embedded copy in the resulting group $\mathcal G'$.

\begin{prop}\label{prop:tfhowson}
The group $\mathcal G'$ is torsion-free, $2$-generated and simple. Moreover, $\mathcal G'$ is a Howson group.
\end{prop}

\begin{proof}
Since all the $G_i, i \in I$ are  torsion-free groups without involutions, the group $\mathcal G'$ is torsion-free, $2$-generated and simple by Corollary \ref{cor ob embedding thm}.

Let $H,K\leq \mathcal G'$ be finitely generated. If $H=K=\mathcal G'$, then $H\cap K=\mathcal G'$ is finitely generated.  Now suppose one of $H$ and $K$ is proper. Hence $H\cap K$ is a proper subgroup of $\mathcal G'$. By Corollary \ref{cor ob embedding thm}, the subgroup $H\cap K$ is either infinite cyclic or contained in a conjugate of some $\widehat D_n$. In the first case it is finitely generated.  In the second case it is also finitely generated because $\widehat D_n$ is polycyclic by Proposition \ref{prop:tfpieces}, and all subgroups of a finitely generated polycyclic group are finitely generated by \cite[Proposition 5.5]{CSD21}. Therefore $\mathcal G'$ is Howson.
\end{proof}

\begin{prop}\label{prop:tfnotstrong}
The group $\mathcal G'$ is not strongly Howson. In fact, $\xi_{\mathcal G'}(m,k)=\infty$ for any integers $m, k\geq 2$.
\end{prop}

\begin{proof}
For every $n\geq 3$, by Proposition~\ref{prop:tfpieces}, there exist subgroups of $\widehat D_n\leq \mathcal G'$ such that
$$ \dgen(\widehat A_n)\leq 2,\qquad \dgen(\widehat B_n)\leq 2,\qquad \dgen(\widehat A_n\cap \widehat B_n)=n.$$
Since $n$ is arbitrary, $\xi_{\mathcal G'}(m,k)=\xi_{\mathcal G'}(2,2)=\infty$ for any $m, k\geq 2$. Therefore, $\mathcal G'$ is not strongly Howson.
\end{proof}



\begin{proof}[\textbf{Proof of Theorem~\ref{thm:torsionfree}}]
This follows from Propositions~\ref{prop:tfhowson} and \ref{prop:tfnotstrong}.
\end{proof}
Since $\mathcal G'$ is infinite and simple, it is not residually finite. As corollary \ref{cor:non-residual finiteness}, we have:
\begin{cor}\label{cor non residually finite for torsion-free}
 The group $\mathcal G'$ is not residually finite, and hence not linear over any field.   
\end{cor}

\section{Questions and further directions}

The two constructions separate the Howson property from the strongly Howson property in the class of finitely generated groups, and also in the torsion-free finitely generated class. 
Both examples constructed here are far from residually finite. The group $\mathcal G$ has no nontrivial finite quotient by Corollary~\ref{cor:non-residual finiteness}, while the group $\mathcal G'$ is infinite and simple, and hence also has no non-trivial finite quotient. This naturally leads to the following questions.

\begin{ques}
Does there exist a finitely generated, residually finite, Howson group which is not strongly Howson?
\end{ques}


\begin{ques}
    Under what conditions are the Howson property and the strongly Howson property equivalent?
\end{ques}

The constructions of $\mathcal G$ and $\mathcal G'$ are still highly non-explicit. It would be interesting to find such examples inside more familiar classes of groups (for example some geometric groups), or to prove that familiar Howson classes satisfy stronger uniform bounds. 

\begin{ques}
Is it possible to construct a torsion-free, finitely generated (or finitely presented) Howson group that fails to be strongly Howson through a more explicit geometric construction, for example without using Obraztsov's embedding theorem?
\end{ques}


\vspace{0.2cm}
\noindent\textbf{Acknowledgment and AI disclosure.} The authors would like to thank Shengkui Ye for helpful discussions. ChatGPT assisted in constructing the example in Section \ref{sect 2}. The proofs in this paper were checked and corrected by the human authors.


\begin{thebibliography}{99}

\bibitem{AMS14} Y. Antol\'in, A. Martino and I. Schwabrow, \emph{Kurosh rank of intersections of subgroups of free products of right-orderable groups}, Math. Res. Lett. {\bf 21} (2014), no.~4, 649--661.

\bibitem{ASS15} V. Ara\'ujo, P.~V. Silva and M. Sykiotis, \emph{Finiteness results for subgroups of finite extensions}, J. Algebra {\bf 423} (2015), 592--614.

\bibitem{Ar98} G.~N. Arzhantseva, \emph{Generic properties of finitely presented groups and Howson's theorem}, Comm. Algebra {\bf 26} (1998), no.~11, 3783--3792.

\bibitem{Ba66} B. Baumslag, \emph{Intersections of finitely generated subgroups in free products}, J. London Math. Soc. {\bf 41} (1966), 673--679.



\bibitem{BB80} R.~G. Burns and A.~M. Brunner, \emph{Two remarks on Howson's group property}, Algebra i Logika {\bf 18} (1979), no.~5, 513--522, 632.

\bibitem{CSD21} T. Ceccherini-Silberstein and M. D’Adderio, ``Polycyclic Groups'' in \emph{Topics in Groups and Geometry}, (2021), Springer Monographs in Mathematics.

\bibitem{Fr14} J. Friedman, \emph{Sheaves on graphs, their homological invariants, and a proof of the Hanna Neumann conjecture: with an appendix by Warren Dicks}, Mem. Amer. Math. Soc. {\bf 233} (2015), no.~1100, xii+106 pp..


\bibitem{Ho54} A.~G. Howson, \emph{On the intersection of finitely generated free groups}, J. London Math. Soc. {\bf 29} (1954), 428--434.

\bibitem{Ka97} I. Kapovich, \emph{Amalgamated products and the Howson property}, Canad. Math. Bull. {\bf 40} (1997), no.~3, 330--340.

\bibitem{Ma40} A. I. Mal'cev, \emph{On isomorphic matrix representations of infinite groups}, Mat. Sb. 8 (1940), 405–422.

\bibitem{Mi12} I. Mineyev, \emph{Submultiplicativity and the Hanna Neumann conjecture}, Ann. of Math. (2) {\bf 175} (2012), no.~1, 393--414.

\bibitem{Ne56} H. Neumann, \emph{On the intersection of finitely generated free groups}, Publ. Math. Debrecen {\bf 4} (1956), 186--189,

\bibitem{Ni13} B. Nica, \emph{Linear groups - Malcev’s theorem and Selberg’s lemma}, 2013, preprint, arXiv:1306.2385.

\bibitem{Ob90} V.~N. Obraztsov, \emph{An embedding theorem for groups and its corollaries}, Math. USSR-Sb. {\bf 66} (1990), no.~2, 541--553; translated from Mat. Sb. {\bf 180} (1989), no.~4, 529--541, 560.

\bibitem{Ob96} V.~N. Obraztsov, \emph{Embedding into groups with well-described lattices of subgroups}, Bull. Austral. Math. Soc. {\bf 54} (1996), no.~2, 221--240.




\bibitem{ZZ26} Q. Zhang and D. Zhao, \emph{Howson groups which are not strongly Howson}, Comm. Algebra {\bf 54} (2026), no.~5, 1902--1907.

\end{thebibliography}
\end{document}